\documentclass[12pt,amsymb,fullpage]{amsart}
\usepackage{amssymb,amscd,pstricks}

\newtheorem{theorem}{Theorem}[section]
\newtheorem{defn}[theorem]{Definition}

\newtheorem{lemma}[theorem]{Lemma}

\newtheorem{eple}[theorem]{Example}
\newtheorem{rmk}[theorem]{Remarks}
\newtheorem{dsc}[theorem]{Discussion}
\newtheorem{nota}[theorem]{Notation}

\newsavebox{\indbin}
\savebox{\indbin}{\begin{picture}(0,0)
\newlength{\gnu}
\settowidth{\gnu}{$\smile$} \setlength{\unitlength}{.5\gnu}
\put(-1,-.65){$\smile$} \put(-.25,.1){$|$}
\end{picture}}

\newcommand{\be}{\begin{enumerate}}
\newcommand{\bd}{\begin{defn}}
\newcommand{\bt}{\begin{theorem}}
\newcommand{\bl}{\begin{lemma}}
\newcommand{\ee}{\end{enumerate}}
\newcommand{\ed}{\end{defn}}
\newcommand{\et}{\end{theorem}}
\newcommand{\el}{\end{lemma}}

\begin{document}
\title{A Proof of the Ergodic Theorem using Nonstandard Analysis}
\author{Tristram de Piro}
\address{Mathematics Department, The University of Exeter, Exeter}
 \email{t.de-piro@exeter.ac.uk}
\maketitle
\begin{abstract}
The following paper follows on from \cite{kam} and gives a rigorous proof of the Ergodic Theorem, using nonstandard analysis.
\end{abstract}
\begin{section}{The Ergodic Theorem}
\noindent \\
There are many versions of the ergodic theorem, but the one we will prove in this paper, using nonstandard analysis, is the following;\\

\begin{theorem}{Ergodic Theorem}
\label{one}

Let $(\Omega,\mathfrak{C},\mu)$ be a probability space, and let $T$ be a measure preserving transformation, then, if $g\in L^{1}(\Omega,\mathfrak{C},\mu)$;\\

$\diamond{g}(\omega)=lim_{n\rightarrow\infty}{1\over n}\sum_{i=0}^{n-1}g(T^{i}\omega)$\\

exists for almost all $\omega\in\Omega$, with respect to $\mu$, and, $\diamond{g}\in L^{1}(\Omega,\mathfrak{C},\mu)$, with;\\

$\int_{\Omega}\diamond{g} d\mu =\int_{\Omega}g d\mu$\\

\end{theorem}

\begin{rmk}
\label{two}
There are a number of good standard proofs of this result. A particular good reference is \cite{Pet}. However, the reader should be aware that it is assumed there that $\mathfrak{C}$ is complete and $T$ is\emph{invertible}, in the sense that $T$ is one-one and onto, and both $T$ and $T^{-1}$ are measurable. A m.p.t is then required to satisfy $\mu(C)=\mu(T^{-1}C)$ for all $C\in{\mathfrak C}$. We will not require these assumption in the proofs of this section, in the sense that we only require a m.p.t to be a measurable $T$ with $\mu(C)=\mu(T^{-1}C)$ for all $C\in{\mathfrak C}$. In \cite{Pet}, a seemingly stonger result is shown, (under the above assumptions), namely that if $C\in\mathfrak{C}$, with $T^{-1}(C)=C$, then;\\

$\int_{C}\diamond{g} d\mu =\int_{C}g d\mu$ $(*)$\\

 from which it easily follows that if $\mathfrak{C'}$ is the sub $\sigma$-algebra of all $T$-invariant sets, where a set $C$ is $T$ invariant in \cite{Pet}, if $T^{-1}C=C$ a.e d$\mu$, then $\diamond{g}=E(g|\mathfrak{C'})$, $(**)$. In the particular case when $T$ is ergodic, that is every $T$ invariant set has measure $0$ or $1$, we obtain the well known result that $\diamond{g}=E(g)$ a.e d$\mu$, $(***)$. However, this result $(*)$ follows easily from our Theorem \ref{one}. as we can, wlog, assume that $\mu(C)>0$, and then restrict and rescale the measure. Of course, we even obtain a slight strengthening of $(*)$, by our weaker assumption on a m.p.t, and obtain similar strengthenings of $(**)$ and $(***)$. (It is not necessary to restrict attention to real valued functions, in the statement of the theorem, the complex version follows immediately from the real case).

\end{rmk}

As usual, we work in an $\aleph_{1}$-saturated model. Let $k\in{^{*}{\mathcal N}_{>0}}$ be infinite, and let $K=\{x\in{^{*}{\mathcal N}}:0\leq x<k\}$. We let ${\mathfrak K}$ be the algebra of all internal subsets of $K$. Observe that as $K$ is hyperfinite, ${\mathfrak K}$ is a hyperfinite $^{*}{\sigma}$-algebra. We let $\nu$ denote the counting measure, defined by setting $\nu(A)={Card(A)\over k}$, for $A\in{\mathfrak K}$. We adopt some of the notation of Section 3 in \cite{dep}, and let $P={^{\circ}\nu}$. By Theorem 3.4, and remarks before Lemma 3.15 of \cite{dep}, $P$ extends uniquely to the completion $\mathfrak{B}$ of the $\sigma$-algebra, $\sigma({\mathfrak K})$, generated by ${\mathfrak K}$. It is clear that $(K,{\mathfrak B},P)$ is a probability space, it is also the Loeb space associated to $(K,{\mathfrak K},\nu)$. We let $\phi:K\rightarrow K$ denote the map defined by;\\

 $\phi(x)=x+1$, if $0\leq x<k-1$\\

 $\phi(x)=0$, if $x=k-1$\\

Clearly, $\phi$ is invertible, internal, preserves the counting measure $\nu$, and $\phi^{-1}(\sigma({\mathfrak K}))=\sigma({\mathfrak K})$. Then $P\circ\phi^{-1}$ defines a measure on $(K,\sigma({\mathfrak K}),P)$, extending $\nu$. By Theorem 3.4(ii) of \cite{dep}, it agrees with $P$. By definition of the completion,  $P\circ\phi^{-1}$ agrees with $P$ on $(K,{\mathfrak B},P)$, so $\phi$, and similarly $\phi^{-1}$ are m.p.t's. We will first prove the following;\\

\begin{theorem}
\label{three}
The ergodic theorem, as stated in Theorem \ref{one}, holds for $(K,{\mathfrak B},P,\phi)$.
\end{theorem}

\begin{proof}
Let $g\in L^{1}(K,{\mathfrak B},P)$, without loss of generality, we can assume that $g\geq 0$. For $x\in K$, we let;\\

$\overline{g}(x)=limsup_{n\rightarrow\infty}{1\over n}\sum_{i=0}^{n-1}g(\phi^{i}x)$\\

$\underline{g}(x)=liminf_{n\rightarrow\infty}{1\over n}\sum_{i=0}^{n-1}g(\phi^{i}x)$\\

In order to prove the theorem, it is sufficient to show that $\overline{g}$ is integrable and;\\

$\int_{K}\overline{g} dP \leq \int_{K} g dP \leq \int_{K}\underline{g} dP$ $(\dag)$\\

Then, as $\underline{g}\leq \overline{g}$, we must have equality in $(\dag)$, so $\underline{g}=\overline{g}$ a.e d$P$, that is $\diamond{g}$ exists a.e d$P$, and;\\

$\int_{K}\diamond{g} dP = \int_{K} g dP$\\

as required.\\

Now let $M\in{\mathcal N}_{>0}$, then, as $\overline{g}$ is ${\mathfrak B}$-measurable, see \cite{Rud}, $min(\overline{g},M)$ is integrable with respect to $P$. Let $\epsilon>0$ be standard, then we can apply Theorem \ref{appone} in the Appendix to this paper, and Definition 3.9 and Remarks 3.10 of \cite{dep}, to obtain internal functions $F,G:K\rightarrow^{*}{\mathcal R}$, with $g\leq F$ and $G\leq min(\overline{g},M)$, such that;\\

$|\int_{A} g dP - {1\over k}{^{*}{\sum_{x\in A}F(x)}}| < \epsilon$\\

$|\int_{A} min(\overline{g},M) dP - {1\over k}{^{*}{\sum_{x\in A}G(x)}}| < \epsilon$, for all internal $A\subset K$, $(\dag\dag)$.\\

Now observe that $\overline{g}$ is $\phi$-invariant,(\footnote{There is a probably a proof of this result in the literature, but we supply one here. Fix $x\in K$. Let $A_{m}={1\over m}\sum_{i=0}^{m-1}g(\phi^{i}x)$ and let $B_{m}={1\over m}\sum_{i=0}^{m-1}g(\phi^{i+1}x)$. Then a simple calculation shows that ${mB_{m}+g(x)\over m+1}=A_{m+1}$. Hence, $|B_{m}-A_{m+1}|=|{A_{m+1}-g(x)\over m}|$, $(*)$. Suppose that $\overline{g}(x)=t<\infty$, $(**)$, (the case when $\overline{g}(x)=\infty$ is similar), and $\overline{g}(\phi x)<t$, $(***)$, (the case $\overline{g}(\phi x)>t$ is again similar). Then, by $(***)$, there exists $\delta>0$, such that, for $m\geq m_{0}$, $B_{m}<t-\delta$. By $(*)$ and $(**)$, we can find $m_{1}\geq m_{0}$, such that $|B_{m}-A_{m+1}|<{\delta\over 2}$, for $m\geq m_{1}$. Again, by $(*)$, we can find $m_{2}\geq m_{1} \geq m_{0}$, such that $A_{m_{2}+1}>t-{\delta\over 2}$. This clearly gives a contradiction.}). Fixing $x\in K$, by the definition of $\overline{g}$, we can find $n\in{\mathcal N}_{>0}$ such that;\\

$min(\overline{g}(x),M)\leq {1\over n}\sum_{i=0}^{n-1}g(\phi^{i}x) + \epsilon$ $(*)$\\

Then, if $0\leq m \leq n-1$, we have;\\

$G(\phi^{m}x)\leq min(\overline{g}(\phi^{m}x),M)$, by definition of $G$\\

$\indent \ \ \ \ \ \ \ \ = min(\overline{g}(x),M)$, by $\phi$ invariance of $\overline{g}$\\

$\indent \ \ \ \ \ \ \ \ \leq {1\over n}\sum_{i=0}^{n-1}g(\phi^{i}x) + \epsilon$, by $(*)$\\

$\indent \ \ \ \ \ \ \ \ \leq {1\over n}\sum_{i=0}^{n-1}F(\phi^{i}x) + \epsilon$, by definition of $F$\\

Therefore,\\

$\sum_{i=0}^{n-1} G(\phi^{i}x) \leq n({1\over n}\sum_{i=0}^{n-1}F(\phi^{i}x) + \epsilon)=\sum_{i=0}^{n-1}F(\phi^{i}x) +n\epsilon$ $(**)$\\

Now let $S_{G}:[1,k)\times K\rightarrow{^{*}{\mathcal R}}$ be defined by;\\

$S_{G}(n,x)={^{*}\sum_{i=0}^{n-1} G(\phi^{i}x)}$\\

and, similarly, define $S_{F}$. By Definition 2.19 of \cite{dep}, and using the facts that $K$ is $*$-finite, and $G,F$ are internal, $S_{G}$ and $S_{F}$ are internal. Then, the relation $(**)$ becomes the internal relation on $[1,k)\times K$, given by $R(n,x)$ iff $S_{G}(n,x)\leq S_{F}(n,x)+n\epsilon$. Using the fact above, that the fibres of $R$ over $K$ are non-empty, by transfer of the corresponding standard result, we can find an internal function $T:K\rightarrow [1,k)$, which assigns
to $x\in K$, the least $n\in [1,k)$, for which $(**)$ holds. Moreover, as we have observed in $(*)$, $T(x)$ is standard, for all $x\in K$. By Lemma 3.11, $r=max_{x\in K}T(x)$ exists and is standard. Now, define $T_{j}$ hyper inductively by;\\

 $T_{0}=0$ and $T_{j}=T_{j-1}+T(T_{j-1})$\\

and let $J$ be the first $j$ such that $k-r\leq T_{j}<k$.(\footnote{This perhaps requires some explanation. Define $I=\{m\in{^{*}{\mathcal N}}_{>0}:\exists !S(dom(S)=[0,m]\wedge S(0)=0\wedge(\forall 1\leq j\leq m)S(j)=S(j-1)+T(S(j-1)_{mod k}))\}$, $(*)$, then it is easy to see that $I$ is internal, $I(1)$ holds, and $I(m)$ implies $I(m+1)$. Applying Lemma 2.12 of \cite{dep}, $I={^{*}{\mathcal N}}_{>0}$. Hence there exists an internal function $f$, defined on ${^{*}{\mathcal N}}_{>0}$, such that $f(m)$ is the unique $S$ satisfying $(*)$. We can then define $T_{j}=f(j)(j)$, and clearly $T_{j}-T_{j-1}\leq r$. Let $V=\{j\in{^{*}{\mathcal N}}_{>0}:T_{j}<k\}$. Then, as $T\geq 1$, $V$ is the interval $[1,t]$ for some infinite $t<k$. Then $k-r\leq T_{t}<k$, otherwise $T_{t+1}<k$. Then $U=\{j\in{^{*}{\mathcal N}}_{>0}:k-r\leq T_{j}<k\}$ is internal and non empty. Therefore, by transfer, it contains a first element $J$.})\\

Observe that $T_{j}$ defines an internal partition of the interval $[0,T_{J-1}]\subset [0,k)$, into $J-1$ blocks of step size $T_{j}-T_{j-1}=T((T_{j-1})$. Hence, we can write;\\

${1\over k}{^{*}\sum_{x=0}^{T_{J}-1} G(x)}={1\over k}{^{*}\sum_{j=0}^{{J-1}}}{^{*}\sum_{i=0}^{T(T_{j})-1} G(\phi^{i}T_{j})}$\\

$\leq {1\over k}{^{*}\sum_{j=0}^{{J-1}}}{^{*}\sum_{i=0}^{T(T_{j})-1} F(\phi^{i}T_{j})+T(T_{j})\epsilon}$  ,by definition of $T$ and $(**)$.\\

Now we can rearrange this last sum as;\\

${1\over k}{^{*}\sum_{x=0}^{T_{J}-1} F(x)} + {\epsilon\over k}{^{*}\sum_{j=0}^{J-1}T(T_{j})}$\\

$\ \ \ \ \ \ \ \ \ \ \  ={1\over k}{^{*}\sum_{x=0}^{T_{J}-1} F(x)} + {T_{J}\epsilon\over k}$\\

$\ \ \ \ \ \ \ \ \ \ \  < {1\over k}{^{*}\sum_{x=0}^{T_{J}-1} F(x)} + \epsilon $\\

using the facts that ${^{*}\sum_{j=0}^{J-1}T(T_{j})}={^{*}\sum_{j=0}^{J-1}(T_{j+1}-T_{j})}=T_{J}$, and $T_{J}<k$. Therefore, we have that;\\

${1\over k}{^{*}\sum_{x=0}^{T_{J}-1} G(x)}< {1\over k}{^{*}\sum_{x=0}^{T_{J}-1} F(x)} + \epsilon $ $(***)$\\

Now, observing that $\nu([T_{J},k))\leq {r\over k}\simeq 0$, as $r$ is standard, we have $P([T_{J},k))=0$. Hence, using $(\dag\dag)$, $(***)$;\\

$\int_{X} min(\overline{g},M) dP = \int_{[0,T_{J})} min(\overline{g},M) dP < {1\over k}{^{*}{\sum_{x=0}^{T_{J}-1}G(x)}} + \epsilon$\\

$< {1\over k}{^{*}\sum_{x=0}^{T_{J}-1} F(x)} + 2\epsilon < \int_{[0,T{J})}g dP + 3\epsilon = \int_{X}g dP + 3\epsilon$\\

Now, letting $M\rightarrow\infty$ and $\epsilon\rightarrow 0$, we can apply the MCT, to obtain;\\

$\int_{X}\overline{g} dP \leq \int_{X}g dP$\\

As $g$ is integrable with respect to $P$, so is $\overline{g}$, and a similar argument to the above demonstrates that $\int_{X}g dP \leq \int_{X}\underline{g} dP$. Therefore, $(\dag)$ is shown and the theorem is proved.

\end{proof}

We now generalise Theorem \ref{three}, to obtain Theorem \ref{one}. We let $\mathcal{P}$ consist of spaces of the form $({\mathcal R}^{\mathcal N}, \mathfrak{D}, \lambda,\sigma)$, where $\mathfrak{D}$ is the Borel field on ${\mathcal R}^{\mathcal N}$, $\sigma$ is the left shift on ${\mathcal R}^{\mathcal N}$, and $\lambda$ is a shift invariant probability measure. Note that $\sigma$ is not invertible, but we require that $\lambda=\sigma_{*}\lambda$, so $\sigma$ is a m.p.t, with respect to $\lambda$. Similarly, we let $\mathcal{Q}$ consist of spaces of the form  $({[0,1]}^{\mathcal N},\mathfrak{E},\rho,\sigma)$, where $\mathfrak{E}$ is the Borel field on ${[0,1]}^{\mathcal N}$, $\sigma$ is again the left shift, and $\rho$ is a shift invariant probability measure.\\

We first require the following simple lemma;\\

\begin{lemma}
\label{four}
Theorem \ref{one} is true iff the Ergodic Theorem holds for all spaces in $\mathcal{P}$.
\end{lemma}

\begin{proof}

One direction is obvious. For the other direction, let $(\Omega,\mathfrak{C},\mu,T)$ and $g\in L^{1}(\Omega,\mathfrak{C},\mu)$ be given. Define a map $\tau:\Omega\rightarrow{\mathcal R}^{\mathcal N}$ by $\tau(\omega)(n)=g(T^{n}\omega)$. Clearly, as $g$ is measurable with respect to $\mathfrak{C}$ and $T$ is a m.p.t, using the definition of the Borel field on ${\mathcal R}^{m}$, for finite $m$, we have that for a cylinder set $U\in \mathfrak{D}$, $\tau^{-1}(U)\in \mathfrak{C}$. By the definition of the Borel field on ${\mathcal R}^{\mathcal N}$, $\tau^{-1}(\mathfrak{D})\subset\mathfrak{C}$, (\footnote{As $\{V\in \mathfrak{D}:\tau^{-1}(V)\in \mathfrak{C}\}$ is a $\sigma$-algebra containing the cylinder sets.}). Let $\lambda$ be the probability measure $\tau_{*}\mu$. Then $\lambda$ is $\sigma$ invariant, as clearly, using the fact that $T$ is a m.p.t, $\lambda=\sigma_{*}\lambda$ on the cylinder sets in $\mathfrak{D}$. Using the definition of the Borel field and Caratheodory's Theorem, we obtain that $\lambda=\sigma_{*}\lambda$. Let $\pi:{\mathcal R}^{\mathcal N}\rightarrow {\mathcal R}$ be the projection onto the $0'th$ coordinate. Then $g=\pi\circ\tau$, and, so $\pi\in L^{1}({\mathcal R}^{\mathcal N}, \mathfrak{D},\lambda)$ by the change of variables formula, (\footnote{This states that if $\tau:(X_{1},\mathfrak{C}_{1},\mu_{1})\rightarrow (X_{2},\mathfrak{C}_{2},\mu_{2})$ is measurable and measure preserving, so $\mu_{2}=\tau_{*}\mu_{1}$, then a function $\theta\in L^{1}(X_{2},\mathfrak{C}_{2},\mu_{2})$ iff $\tau^{*}\theta\in L^{1}(X_{1},\mathfrak{C}_{1},\mu_{1})$ and $\int_{C} \theta d\tau_{*}\mu_{1}=\int_{\tau^{-1}(C)} \tau^{*}\theta d\mu_{1}$.}). Moreover, $g(T^{i}\omega)=\pi(\sigma^{i}\tau(\omega))$, so applying the Ergodic Theorem for $({\mathcal R}^{\mathcal N},\mathfrak{D}, \lambda,\sigma)$, with the change of variables formula, we have that  $\diamond{g}$ exists and $\diamond{g}=\diamond{\pi}\circ\tau$ a.e $d\mu$, and $\int_{\Omega} \diamond{g} d\mu = \int_{\Omega} (\diamond{\pi}\circ\tau) d\mu = \int_{{\mathcal R}^{\mathcal N}} \diamond{\pi} d\lambda = \int_{{\mathcal R}^{\mathcal N}} \pi d\lambda = \int_{\Omega} g d\mu$ as required.

\end{proof}

We make the following definition;\\

\begin{defn}
\label{five}
We say that $({\mathcal R}^{\mathcal N}, \mathfrak{D}, \lambda,\sigma)\in\mathcal{\mathcal P}$ is a factor of $(K,{\mathfrak B},P,\phi)$ if there exists;\\

$\Gamma:(K,{\mathfrak B},P)\rightarrow({\mathcal R}^{\mathcal N}, \mathfrak{D},\lambda)$\\

which is measurable and measure preserving, such that;\\

 $\Gamma(\phi x)=\sigma(\Gamma x)$ a.e $(x\in K)$ d$P$.\\

We make the same definition if $({[0,1]}^{\mathcal N},\mathfrak{E},\rho,\sigma)\in\mathcal{\mathcal Q}$.\\

\end{defn}

\begin{lemma}
\label{six}
Suppose that $({\mathcal R}^{\mathcal N}, \mathfrak{D}, \lambda,\sigma)\in\mathcal{\mathcal P}$ is a factor of $(K,{\mathfrak B},P,\phi)$, then, if the Ergodic Theorem holds for $(K,{\mathfrak B},P,\phi)$, it holds for $({\mathcal R}^{\mathcal N}, \mathfrak{D}, \lambda,\sigma)$.
\end{lemma}

\begin{proof}
The proof is similar to Lemma \ref{four}. If $h\in L^{1}({\mathcal R}^{\mathcal N}, \mathfrak{D},\lambda)$, then, by change of variables, $\Gamma^{*}h\in L^{1}(K,{\mathfrak B},P)$. Applying the Ergodic Theorem for $(K,{\mathfrak B},P,\phi)$ and the definition of a factor, we have that $\diamond{\Gamma^{*}h}$ exists and $\diamond{\Gamma^{*}h}=\Gamma^{*}\diamond{h}$,  a.e d$P$, $(*)$. So $\diamond{h}$ exists a.e d$\lambda$, and, again, by change of variables, $(*)$, and the Ergodic theorem for $(K,{\mathfrak B},P,\phi)$;\\

$\int_{{\mathcal R}^{\mathcal N}}\diamond{h} d\lambda= \int_{K}\Gamma^{*}(\diamond{h}) dP = \int_{K}\diamond{(\Gamma^{*}h)} dP = \int_{K}(\Gamma^{*}h) dP = \int_{{\mathcal R}^{\mathcal N}}h d\lambda$\\

\end{proof}

We now claim the following;\\

\begin{lemma}
\label{seven}
Every space in $\mathcal{P}$ is isomorphic, in the sense of dynamical systems, (\footnote{By which I mean there exists measurable and measure preserving maps $r:({\mathcal R}^{\mathcal N}, \mathfrak{D}, \lambda)\rightarrow ({[0,1]}^{\mathcal N},\mathfrak{E},\rho)$ and $s:({[0,1]}^{\mathcal N},\mathfrak{E},\rho)\rightarrow({\mathcal R}^{\mathcal N}, \mathfrak{D}, \lambda)$ such that $s\circ r=Id$ and $r\circ\sigma=\sigma\circ r$ a.e d$\lambda$, $r\circ s=Id$  and $s\circ\sigma=\sigma\circ s $ a.e d$\rho$}), to a space in $\mathcal{Q}$.
\end{lemma}

\begin{proof}
There exists an isomorphism, in the sense of measure spaces,  $\Phi:({\mathcal R}^{\mathcal N}, \mathfrak{D}, \lambda)\rightarrow ([0,1],\mathfrak{E}',\rho')$, where $\mathfrak{E}'$ is the Borel field and $\rho'$ is a probability measure, see \cite{Pet}, Theorem 1.4.4. Now define $r:{\mathcal R}^{\mathcal N}\rightarrow[0,1]^{\mathcal N}$ by $r(\omega)(n)=\Phi(\sigma^{n}\omega)$. Again, using the argument above and the fact that $\Phi$ and $\sigma$ are measurable, $r^{-1}(\mathfrak{E})\subset \mathfrak{D}$, where is the Borel field on $[0,1]^{\mathcal N}$. Let $\rho$ be the probability measure $r_{*}\lambda$, so $r:({\mathcal R}^{\mathcal N}, \mathfrak{D}, \lambda)\rightarrow({[0,1]}^{\mathcal N},\mathfrak{E},\rho)$ is also measure preserving. We have that $r(\sigma\omega)(n)=\Phi(\sigma^{n+1}\omega)=(r\omega)(n+1)=\sigma(r\omega)(n)$, so $r\circ\sigma=\sigma\circ r$, for all $\omega\in{\mathcal R}^{\mathcal N}$. This also shows that $\rho$ is $\sigma$ invariant, as $\lambda$ is $\sigma$ invariant. Hence, $({[0,1]}^{\mathcal N},\mathfrak{E},\rho,\sigma)$ belongs to $\mathcal{Q}$. Define $s:({[0,1]}^{\mathcal N},\mathfrak{E},\rho)\rightarrow({\mathcal R}^{\mathcal N}, \mathfrak{D},\lambda)$, by, $s(\omega')=\Phi^{-1}(\pi(\omega'))$, where again $\pi$ is the $0$'th coordinate projection, clearly $s$ is measurable. Then $(s\circ r)(\omega)=\Phi^{-1}\circ\pi\circ r(\omega)$, and $\pi\circ r(\omega)=r(\omega)(0)=\Phi(\omega)$, so $(s\circ r)=Id$ a.e, and, similarly $r\circ\sigma=\sigma\circ r$ a.e d$\lambda$. This clearly shows that $s$ is measure preserving, and that $(r\circ s)=Id$, $s\circ\sigma=\sigma\circ s$,$(*)$, hold, restricted to $r(U)$, where $\lambda(U)=1$. As, by definition, $\rho(\lambda(U))=1$, and the conditions in $(*)$ are measurable, we obtain the result. (Note that the map $s$ need not be invertible in the ordinary sense.)
\end{proof}

We now make the following;\\

\begin{defn}
\label{eight}
Let $({[0,1]}^{\mathcal N},\mathfrak{E},\rho,\sigma)$ belong to ${\mathcal Q}$, then we say that $\alpha$ is typical for $\rho$ if;\\

$lim_{n\rightarrow\infty}\sum_{i=0}^{n-1}g(\sigma^{i}\alpha)=\int_{{[0,1]}^{\mathcal N}}g d\rho$\\

for any $g\in C([0,1]^{\mathcal N})$.\\

\end{defn}

We now show;\\

\begin{theorem}
\label{nine}
Let $({[0,1]}^{\mathcal N},\mathfrak{E},\rho,\sigma)$ belong to ${\mathcal Q}$, possessing a typical element $\alpha$. Then $({[0,1]}^{\mathcal N},\mathfrak{E},\rho,\sigma)$ is a factor of $(K,{\mathfrak B},P,\phi)$ in the sense of Definition \ref{five}.
\end{theorem}

\begin{proof}

Define $\Gamma:K\rightarrow {[0,1]}^{\mathcal N}$ by $\Gamma(x)={^{\circ}(\sigma^{x}\alpha)}$, (\footnote{Here,  $(\sigma^{x}\alpha)={^{*}H}(x)$ for the internal function $^{*}H:{^{*}{\mathcal N}}\rightarrow {^{*}([0,1]^{\mathcal N}) }=(^{*}[0,1])^{^{*}{\mathcal N}}$, obtained by transferring the standard function $H:{\mathcal N}\rightarrow [0,1]^{\mathcal N}$, defined by $H(n)=\sigma^{n}(\alpha)$. Observe that $[0,1]^{\mathcal N}$ is compact and Haussdorff in the product topology, so, by Theorem 2.34 of \cite{dep}, there exists a unique standard part mapping $^{\circ}:{^{*}([0,1]^{\mathcal N})}\rightarrow [0,1]^{\mathcal N}$. In fact, see \cite{Rob}, this mapping is defined by setting $^{\circ}s={(^{\circ}s(n))}_{n\in{\mathcal N}}$ where $s:{^{*}{\mathcal N}}\rightarrow ^{*}[0,1]$ is internal.}). Now suppose that $g\in C([0,1]^{\mathcal N})$, so, as $[0,1]^{\mathcal N}$ is compact, $g$ is bounded,$(*)$, then;\\

$^{\circ}g(\sigma^{x}\alpha)=g(\Gamma(x))$ for all $x\in K$, $(**)$ (\footnote{I have also denoted by $g$, the transfer of $g$ to $^{*}C({^{*}([0,1]^{\mathcal N})})$. Observe that $\sigma^{x}(\alpha)\simeq\Gamma(x)$ by definition of $\Gamma$, it is then straightforward to adapt Theorem 2.25 of \cite{dep}, using the fact that $g$ is continuous, to show that $g(\sigma^{x}\alpha)\simeq g(\Gamma(x))$.}).\\

This implies that $\Gamma$ is measurable, as if $B$ is an open set for the product topology on $[0,1]^{\mathcal N}$, then, taking $g$ to be a continuous function with support $B$, $\Gamma^{*}{g}$ is measurable with respect to $P$, by Theorem 3.8 (Lemma 3.15) of \cite{dep}. This clearly implies that $\Gamma^{-1}(B)$ is measurable. By previous arguments, we obtain the result. Moreover;\\

$\int_{{[0,1]}^{\mathcal N}}g d\rho$\\

$=lim_{n\rightarrow\infty}{1\over n}\sum_{i=0}^{n-1}g(\sigma^{i}\alpha)$, (by definition of a typical element $\alpha$)\\

$=^{\circ}({1\over k}{^{*}\sum_{x=0}^{k-1}g(\sigma^{x}\alpha)})$, (\footnote{Observe that $s(n)={1\over n}\sum_{i=0}^{n-1}g(\sigma^{i}\alpha)$ is a standard sequence, with limit $s=\int_{{[0,1]}^{\mathcal N}}g d\rho$. By Theorem 2.22 of \cite{dep}, using the fact that $k$ is infinite, $s\simeq s(k)$. Using Definition 2.19 of \cite{dep}, it is clear that $s(k)$ is the hyperfinite sum ${1\over k}{^{*}\sum_{x=0}^{k-1}g(\sigma^{x}\alpha)}$}).\\

$={^{\circ}\int_{K}g(\sigma^{x}\alpha) d\nu}$ (using Definition 3.9 of \cite{dep} and Remarks 3.10 of \cite{dep})\\

$=\int_{K}g(\Gamma(x)) dP$, (using $(*)$, $(**)$ and Theorem 3.12 of \cite{dep} (Lemma 3.15 of \cite{dep}))\\

 $(***)$\\

The result of $(***)$ implies that $\Gamma$ is measure preserving. The probability measure $\Gamma_{*}P$ defines a bounded linear functional on $C([0,1]^{\mathcal N})$, which agrees with $\rho$. Using the fact that $[0,1]^{\mathcal N}$ is a compact Hausdorff space, and $\rho, \Gamma_{*}P$ are regular, see \cite{Rud} Theorem 2.18, (\footnote{It is easy to see that $[0,1]^{\mathcal N}$ is $\sigma$-compact. This follows from the fact that finite intersections of cylinder sets form a basis for the topology on $[0,1]^{\mathcal N}$. Any open set in $U$ in ${[0,1]}^{m}$ is a countable union of closed sets, as every $x\in U$ lies inside a closed box $B$ with rational corners, such that $B\subset U$. Hence, any cylinder set is a countable union of such closed sets $\pi_{m}^{-1}(B)$.}), we can apply the uniqueness part of the Riesz Representation Theorem, see \cite{Rud} Theorem 6.19, to conclude that $\Gamma_{*}P=\rho$, we will discuss this further below. Now, as $\sigma$ is continuous with respect to $\mathfrak{E}$, (\footnote{Again I have denoted by $\sigma$ the transfer of the standard shift $\sigma$ to ${^{*}([0,1]^{\mathcal N})}$. The fact that $\sigma(\sigma^{x}\alpha)=\sigma^{x+1}(\alpha)$ follows immediately by transferring the standard fact that $\sigma(\sigma^{n}(\alpha))=\sigma^{n+1}(\alpha)$ for $n\in{\mathcal N}$.}),;\\

$\sigma(\Gamma x)=\sigma(^{\circ}(\sigma^{x}\alpha)=^{\circ}(\sigma(\sigma^{x}\alpha))=^{\circ}(\sigma^{x+1}\alpha)=\Gamma(x+1)=\Gamma(\phi(x))$\\

except for $x=k-1$, so a.e $dP$. Hence, the result follows.

\end{proof}

We now address the problem of finding a typical element for a space $({[0,1]}^{\mathcal N},\mathfrak{E},\rho,\sigma)\in {\mathcal Q}$. By Theorem \ref{three}, Lemma \ref{four}, Lemma \ref{six}, Lemma \ref{seven} and Theorem \ref{nine}, we then obtain the Ergodic Theorem \ref{one}. The proof of this result does \emph{not} require the Ergodic Theorem, and is originally due to de Ville, see \cite{kam}.

\begin{defn}
\label{ten}
We say that a sequence of measures $(\rho_{n})_{n\in \mathcal N}$ converges weakly to $\rho$ if, for all $g\in C([0,1]^{\mathcal N})$;\\

$lim_{n\rightarrow\infty}(\int_{[0,1]^{\mathcal N}}g d\rho_{n}) = \int_{[0,1]^{\mathcal N}}g d\rho$.\\

\end{defn}

We require the following lemma;\

\begin{lemma}
\label{eleven}
Let ${(\alpha_{n})}_{n\in\mathcal{N}}$ be a sequence of periodic, with respect to $\sigma$, elements in $[0,1]^{\mathcal N}$, such that the sequence of probability measures $(\rho_{\alpha_{n}})_{n\in\mathcal{N}}$ converges weakly to $\rho$, where;\\

$\rho_{\alpha_{n}}={1\over c_{n}}(\delta_{\alpha_{n}}+\delta_{\sigma\alpha_{n}}+\ldots+\delta_{\sigma^{c_{n}-1}\alpha_{n}})$\\

$\delta_{\alpha_{n}}$ denotes the probability measure supported on ${\alpha_{n}}$ and $c_{n}$ denotes the period of $\alpha_{n}$. Then there exists a sequence ${(r_{n})}_{n\in\mathcal{N}}$ of positive integers, such that if ${(T_{n})}_{n\in\mathcal{N}}$ is defined by $T_{0}=0$ and $T_{n+1}-T_{n}=c_{n}r_{n}$, the element  $\alpha\in[0,1]^{\mathcal N}$, defined by $\alpha(m)=\alpha_{n}(m-T_{n})$, for $T_{n}\leq m <T_{n+1}$, is typical for $\rho$.

\end{lemma}

\begin{proof}
The proof is intuitively clear, but hard to write down rigorously. As $\rho_{\alpha_{n}}$ converges weakly to $\rho$, we have that;\\

$lim_{n\rightarrow\infty}(\int_{X} f d \rho_{\alpha_{n}})=\int_{X}f d \rho$\\

By definition of $\rho_{\alpha_{n}}$;\\

$\int_{X} f d \rho_{\alpha_{n}}={1\over c_{n}}(f(\alpha_{n})+\ldots+f(\sigma^{c_{n}-1}\alpha_{n})$\\

So it is sufficient to prove that;\\

$lim_{n\rightarrow\infty}{1\over n}\sum_{i=0}^{n-1} f(\sigma^{i}\alpha)=lim_{n\rightarrow\infty}{1\over c_{n}}(f(\alpha_{n})+\ldots+f(\sigma^{c_{n}-1}\alpha_{n}))$ $(*)$\\

We first claim that, if $f\in C([0,1]^{\mathcal{N}})$, there exists an increasing sequence $\{m_{n}\}_{n\in{\mathcal N}}$ of positive integers, such that if $b,c\in{[0,1]^{\mathcal N}}$, and agree up to the $m_{n}$'th coordinate, then $|f(b)-f(c)|<{1\over n}$, $(**)$. In order to see this, for $x\in {[0,1]^{\mathcal N}}$, let $U_{x}=\{y:|f(x)-f(y)|<{1\over 2n}\}$. As $f$ is continuous, $U_{x}$ is open in the Borel field, hence there exists $V_{x}\subset U_{x}$, containing $x$,  of the form $\pi^{-1}(W_{x})$, where $W_{x}\subset {\mathcal R}^{n_{x}}$ is open, and $\pi$ is the projection onto the first $n_{x}$ coordinates. Then, if $y,z\in U_{x}$, $|f(y)-f(z)|\leq |f(y)-f(x)|+|f(z)-f(x)|<{1\over n}$. The sets $\{V_{x}:x\in X\}$ form an open cover of ${[0,1]^{\mathcal{N}}}$, which is compact in the product topology. Hence, there exists a finite subcover $V_{x_{1}}\cup\ldots\cup V_{x_{r}}$. We can choose $m_{n}$ such that each $V_{x_{j}}$ is of the form $\pi^{-1}(W_{x_{j}})$, for $W_{x_{j}}\subset {\mathcal R}^{m_{n}}$. Then, if $b$ and $c$ agree up to the $m_{n}$'th coordinate, we have that $b\in V_{x_{j}}$ iff $c\in V_{x_{j}}$, so $|f(b)-f(c)|<{1\over n}$, showing $(**)$. Now let $\{g_{n}\}_{n\in\mathcal{N}}$ be any increasing sequence of positive integers, such that if $Q_{n}=sup\{|f(b)-f(c)|: \pi_{g_{n}}(b)=\pi_{g_{n}}(c)\}$, then $\{Q_{n}\}_{n\in\mathcal{N}}$ is decreasing and $lim_{n\rightarrow\infty}Q_{n}=0$. Clearly such a sequence exists by $(**)$. Without loss of generality, we can choose $\{g_{n}\}_{n\in\mathcal{N}}$, such that the periods $c_{n}|g_{n}$, $(\sharp)$. Now choose $\{T_{i}\}_{i\in\mathcal{N}}$ as follows;\\

$(i)$. $T_{i+1}\geq 2^{i}T_{i}$\\

$(ii)$. $g_{i}|T_{i+1}-T_{i}$  (so $c_{i}|T_{i+1}-T_{i}$)\\

$(iii)$. $C_{i}={T_{i+1}-T_{i}\over g_{i}}\gneq C_{i-1}={T_{i}-T_{i-1}\over g_{i-1}}$ $(i\geq 1)$.\\

$(iv)$. $T_{i}\geq 2^{i}c_{i}$ $(i\geq 1)$.\\

We now claim there exists a decreasing sequence $\{b_{n}\}_{n\in{\mathcal N}_{>0}}$ of positive reals, such that;\\

$|{1\over T_{n}}\sum_{i=0}^{T_{n}-1}f(\sigma^{i}\alpha) - t_{n}| \leq b_{n}$ $(***)$\\

where $lim_{n\rightarrow\infty}b_{n}=0$, and $t_{n}={1\over c_{n}}(f(\alpha_{n})+\ldots+f(\sigma^{c_{n}-1}\alpha_{n})$, for $n\geq 1$. For ease of notation, we let;\\

$A_{n}={1\over n}\sum_{i=0}^{n-1} f(\sigma^{i}\alpha)$\\

$A_{m,n}={1\over n-m}\sum_{i=m}^{n-1}f(\sigma^{i}\alpha)$\\

Recall the law of weighted averages, $A_{n}={mA_{m}+(n-m)A_{m,n}\over n}$. We first estimate $|A_{T_{n}}-A_{T_{n-1},T_{n}}|$. We have;\\

$A_{T_{n}}={T_{n-1}A_{T_{n-1}}+(T_{n}-T_{n-1})A_{T_{n-1},T_{n}}\over T_{n}}$\\

$|A_{T_{n}}-A_{T_{n-1},T_{n}}|$\\

$=|{T_{n-1}\over T_{n}}A_{T_{n-1}} + {T_{n}-T_{n-1}\over T_{n}}A_{T_{n-1},T_{n}} - A_{T_{n-1},T_{n}}|$\\

$\leq {|A_{T_{n}-1}|\over 2^{n-1}} + {|A_{T_{n-1},T_{n}}|\over 2^{n-1}}$ by $(i)$\\

$\leq {M\over 2^{n-2}}$, where $|f|\leq M$,   $(A)$\\

We now estimate the average $A_{T_{n-1},T_{n}}$. The idea is to divide the interval between $T_{n-1}$ and $T_{n}$ into $C_{n-1}$ blocks of length $g_{n-1}$, where the period $c_{n-1}|g_{n-1}$, using $(\sharp)$ and $(ii)$. We estimate $|A_{T_{n-1},T_{n}}-A_{T_{n-1},T_{n}-g_{n}}|$;\\

$A_{T_{n-1},T_{n}}={C_{n-1}-1\over C_{n-1}}A_{T_{n-1},T_{n}-g_{n-1}} + {1\over C_{n-1}}A_{T_{n}-g_{n-1},T_{n}}$\\

$|A_{T_{n-1},T_{n}}-A_{T_{n-1},T_{n}-g_{n-1}}|$\\

$=|{A_{T_{n}-g_{n-1},T_{n}}\over C_{n-1}} - {A_{T_{n-1},T_{n}-g_{n-1}}\over C_{n-1}}| \leq {2M\over C_{n-1}}$ $(B)$\\

We now let;\\

$B_{T_{n-1},m}={1\over m-T_{n-1}}\sum_{i=0}^{m-T_{n-1}-1} f(\sigma^{i}\alpha_{n-1})$, for $m\leq n$.\\

 We estimate $|A_{T_{n-1},T_{n}-g_{n-1}}-B_{T_{n-1},T_{n}-g_{n}}|$. We have that $\sigma^{T_{n-1}+i}\alpha$ and $\sigma^{i}\alpha_{n-1}$ agree up to the $g_{n-1}$'th coordinate, for $0\leq i < T_{n}-T_{n-1}-g_{n-1}$. Therefore, for such $i$, $|f(\sigma^{i}\alpha_{n-1})-f(\sigma^{T_{n-1}+i}\alpha)|\leq Q_{n-1}$, and so;\\

$|A_{T_{n-1},T_{n}-g_{n-1}}-B_{T_{n-1},T_{n-1}-g_{n-1}}| \leq Q_{n-1}$   $(C)$\\

Now, by the same argument as in $(B)$;\\

$|B_{T_{n-1},T_{n}}-B_{T_{n-1},T_{n}-g_{n}}| \leq {2M\over C_{n-1}}$   $(D)$\\

Finally, by periodicity;\\

$B_{T_{n-1},T_{n}}={1\over c_{n-1}}(f(\alpha_{n-1})+\ldots+ f(\sigma^{c_{n-1}-1}\alpha_{n-1}))=t_{n}$   $(E)$\\

Now, combining the estimates $(A),(B),(C),(D),(E)$, we have;\\

$|A_{T_{n}}-t_{n}| \leq {M\over 2^{n-2}} +{2M\over 2^{n-2}} + Q_{n-1} + {2M\over C_{n-1}}=b_{n}$\\

Clearly $\{b_{n}\}_{n\in\mathcal{N}}$ is decreasing. Moreover, $lim_{n\rightarrow\infty} b_{n}=0$, as $lim_{n\rightarrow\infty} C_{n}=\infty$, $(iii)$, and by the choice of $\{Q_{n}\}_{n\in\mathcal{N}}$. This shows $(***)$.
We now have to estimate the averages up to place between the critical points $T_{n}$ and $T_{n+1}$.\\

Case 1. The place $v$ is a periodic point of the form;\\

$T_{n}+mg_{n}$, where $0\leq m\leq C_{n}-1$\\

We have $A_{v}=\lambda A_{T_{n}}+(1-\lambda)A_{T_{n},v}$ $(0\leq\lambda\leq 1)$, where $|A_{T_{n},v}-t_{n+1}|\leq Q_{n}$, by $(C),(E)$, and $|A_{T_{n}}-t_{n}|\leq b_{n}$, by $(***)$. Now, let $t=lim_{n\rightarrow\infty}t_{n}$. Given $\epsilon>0$, choose $N(\epsilon)$, such that $|t_{n}-t|<\epsilon$, for all $n\geq N(\epsilon)$. Then;\\

$|A_{v}-t|\leq max\{|A_{T_{n}}-t|,|A_{T_{n},v}-t|\}$\\

$\leq max\{b_{n}+{\epsilon\over 2}, Q_{n}+{\epsilon\over 2}\}$\\

Choose $N_{1}(\epsilon)\geq N(\epsilon)$, such that $max\{b_{n},Q_{n}\}<{\epsilon\over 2}$, for all $n\geq N_{1}(\epsilon)$, then $|A_{v}-t|<\epsilon$, for all $n\geq N_{1}(\epsilon)$.\\

Case 2. The place $v$ is a possibly non-periodic point of the form;\\

 $T_{n}+w$, where $0\leq w\leq T_{n+1}-T_{n-1}-g_{n}$.\\

Choose periodic points $v_{1}$ and $v_{2}$, with $T_{n}\leq v_{1}\leq v\leq v_{2}\leq T_{n+1}-g_{n}$, and $v_{2}-v_{1}=c_{n}$, so $0\leq v-v_{1}=e\leq c_{n}$. Then $A_{v}={v_{1}\over v_{1}+e}A_{v_{1}}+{e\over v_{1}+e}A_{v_{1},v}$. As $v_{1}\geq T_{n}$, we have;\\

${e\over v_{1}+e}\leq {e\over T_{n}+e}\leq {c_{n}\over T_{n}}\leq {1\over 2^{n}}$ by $(iv)$.\\

Therefore;\\

$|A_{v}-A_{v_{1}}|=|(1-\delta)A_{v_{1}}+\delta A_{v_{1},v}-A_{v_{1}}|$, $(\delta\leq {1\over 2^{n}})$\\

$\leq \delta(|A_{v_{1}}|+|A_{v_{1},v}|)\leq {M\over 2^{n-1}}$\\

For $n\geq N_{1}({\epsilon\over 2})$, $|A_{v_{1}}-t|<{\epsilon\over 2}$, by Case 1, so $|A_{v}-t|<\epsilon$, for $n\geq N_{2}(\epsilon)$, where $N_{2}(\epsilon)=max\{N_{1}({\epsilon\over 2}), log({2M\over\epsilon})+2\}$.\\

Case 3. The place $v$ is of the form;\\

$T_{n}+w$, where $T_{n+1}-T_{n}-g_{n}\leq w\leq T_{n+1}-T_{n}$.\\

We have;\\

$A_{v}=\lambda A_{T_{n}}+(1-\lambda)A_{T_{n},v}$, $(0\leq\lambda\leq 1)$, $(\dag)$,\\

$A_{T_{n},T_{n+1}}=\mu A_{T_{n},v}+(1-\mu) A_{v,T_{n+1}}$, ${C_{n}-1\over C_{n}}\leq\mu\leq 1$\\

 Therefore;\\

$|A_{T_{n},T_{n+1}}-A_{T_{n},v}| \leq {2M\over C_{n}}$\\

$|A_{T_{n},T_{n+1}}-t_{n+1}|\leq b_{n+1}$, by $(B),(C),(D),(E)$\\

$|A_{T_{n},v}-t_{n+1}|\leq {2M\over C_{n}}+b_{n+1}$\\

$|A_{T_{n}}-t_{n}|\leq b_{n}$ by $(***)$\\

$|A_{v}-t|\leq max\{|A_{T_{n}}-t|,|A_{T_{n},v}-t|\}$ by $(\dag)$\\

$\leq max\{b_{n}+|t_{n}-t|,{2M\over C_{n}}+b_{n+1}+|t_{n+1}-t|\}$, $(\dag\dag)$\\

We have, for $n\geq N({\epsilon\over 2})$, $max\{|t_{n}-t|,|t_{n+1}-t|\}<{\epsilon\over 2}$. Choose $N_{3}(\epsilon)$, such that $max\{b_{n},{2M\over C_{n}}+b_{n+1}\}<{\epsilon\over 2}$, for all $n\geq N_{3}(\epsilon)$. Then, for $n\geq N_{3}(\epsilon)$, $|A_{v}-t|<\epsilon$.\\

To complete the proof, let $N_{4}(\epsilon)=max\{N_{1}(\epsilon),N_{2}(\epsilon),N_{3}(\epsilon)\}$. Then, for $n\geq N_{4}(\epsilon)$, $|A_{m}-t|<\epsilon$, for all $m\geq T_{n}$, by Cases 1,2 and 3. Therefore;\\

$lim_{m\rightarrow\infty}{1\over m}\sum_{i=0}^{m-1}f(\sigma^{i}\alpha)=\int_{X} f d\rho$\\

so $\alpha$ is typical, as required.

\end{proof}

We now formulate the following criteria. \\

\begin{lemma}
\label{twelve}
Suppose that for every $g\in C([0,1]^{\mathcal N})$, and $\epsilon>0$, there exists a periodic element $\beta\in [0,1]^{\mathcal N}$, with;\\

$|\int_{[0,1]^{\mathcal N}}g d\rho_{\beta} - \int_{[0,1]^{\mathcal N}}g d\rho| < \epsilon$\\

then there exists a sequence of periodic elements ${(\alpha_{n})}_{n\in\mathcal{N}}$, with $(\rho_{\alpha_{n}})_{n\in\mathcal{N}}$ converging weakly to $\rho$.

\end{lemma}

\begin{proof}
We abbreviate ${[0,1]}^{\mathcal N}$ to $X$. Let $\mathcal{M}$ denote the vector space of real valued regular measures on $(X,\mathfrak{E})$. As we observed every probability measure belongs to $\mathcal{M}$. $\mathcal{M}$ is a Banach space, with norm defined by total variation, see \cite{Rud}. Using the Riesz Representation Theorem, $\mathcal{M}$ can be identified with the dual space $C(X)^{*}$. It is easy to see that then $\mathcal{M}\cong C(X)^{*}$, as Banach spaces, however, we will not require this fact. The weak $*$-topology, see \cite{boll}, on $\mathcal{M}$, is the coursest topology for which all the elements $\hat{g}\in C(X)^{**}$, where $g\in C(X)$, are continuous. Formally, we define a set $U\subset{\mathcal M}$ to be open if for all $\rho\in U$, there exist $\{g_{1},\ldots,g_{n}\}\subset C(X)$, and positive reals $\{\epsilon_{1},\ldots,\epsilon_{n}\}$ such that;\\

$\{\rho'\in {\mathcal M}: |\rho'(g_{i})-\rho(g_{i})| < \epsilon_{i}\}\subset U$\\

Fixing $\rho$, let $\Omega_{\rho}$ denote the open sets containing $\rho$. We show that $\Omega_{\rho}$ has a countable base, $(*)$. Using the compactness argument, given in Lemma \ref{eleven}, and the Stone-Weierstrass Theorem, see \cite{boll}, it is easy to show that the space $V$ of pullbacks of polynomial functions on $[0,1]^{n}$, for some $n$, is dense in $C(X)$. Clearly $V$ has a countable basis, which shows that $C(X)$ is separable, that is, contains a countable dense subset $Y$. Now suppose that $g\in C(X)$, $\epsilon>0$. Let $U_{g,\epsilon}=\{\rho':|\rho'(g)-\rho(g)| < \epsilon\}$, and $D\in{\mathcal Q}$. Choose $\delta\in{\mathcal Q}$ with $\delta<{\epsilon\over 2(D+2|\rho(X)|)}$, and $\gamma\in{\mathcal Q}$ with $\gamma <{\epsilon\over 2}$. Choose $h\in Y$ with $||g-h||_{C(X)}<\delta$. Then $U_{h,\gamma}\cap U_{1,D}\subset U_{g,\epsilon}$, $(**)$, as if $|\rho'(h)-\rho(h)|<\gamma$, then;\\

$|\rho'(g)-\rho(g)|=|\rho'(g-h)+\rho'(h)-\rho(g-h)-\rho(h)|\leq \delta(|\rho'(X)|+|\rho(X)|) + \gamma $\\

and, if $|\rho'(1)-\rho(1)|<D$, then $|\rho'(X)|+|\rho(X)|<D+2|\rho(X)|$, so $|\rho'(g)-\rho(g)|<\epsilon$. This clearly shows $(**)$. As sets of the form $U_{h,q}\in\Omega_{\rho}$, for $h\in Y$, and $q\in{\mathcal Q}$, are countable, we clearly have $(*)$. Let $I:\mathcal{N}\rightarrow\Omega_{\rho}$ be an enumeration of the sets $U_{h,q}$, and let $J:\mathcal{N}\rightarrow\Omega_{\rho}$ define the intersection of the first $n$ elements in $I$. If the assumption in the lemma is satisfied, we can define a sequence of probability measures $(\rho_{\alpha_{n}})_{n\in\mathcal{N}}$, by taking $\rho_{\alpha_{n}}$ to lie inside the open set $J(n)$. Then clearly such a sequence converges to $\rho$ in the weak $*$-topology, hence, for any $g\in C(X)$, as $g$ is continuous for this topology $lim_{n\rightarrow\infty}\rho_{\alpha_{n}}(g)=\rho(g)$. Therefore, the sequence $(\rho_{\alpha_{n}})_{n\in\mathcal{N}}$ converges weakly to $\rho$.
\end{proof}

We refine this criteria further;\\

\begin{defn}
\label{thirteen}
Given a positive integer $m$, we define the partition $E_{m}$ of $[0,1]$ to consist of the sets;\\

$E_{j,m}=[{j\over m},{j+1\over m})$ for $j$ an integer between $0$ and $m-2$\\

$E_{m-1,m}=[{m-1\over m},1]$\\

Given positive integers $m,n$, we define the partition $B_{m,n}$ of $[0,1]^{n}$ to consist of the sets;\\

$B_{\bar j,m,n}=E_{j_{0},m}\times E_{j_{1},m}\times\ldots\times E_{j_{n-1},m}$\\

where $\bar j=(j_{0},j_{1},\ldots,j_{n-1})$ and $\{j_{0},\ldots,j_{n-1}\}$ are integers between $0$ and $m-1$.\\

We define the partition $C_{m,n}$ of $[0,1]^{\mathcal N}$ to consist of the sets;\\

$C_{\bar j,m,n}=\pi_{n}^{-1}(B_{\bar j,m,n})$\\

where $\pi_{n}$ is the projection onto the first $n$ coordinates.

\end{defn}

\begin{lemma}
\label{fourteen}
Let $\epsilon>0$, $g\in C(X)$ be given as in Lemma \ref{twelve}, and let $\rho'$ be a regular Borel measure, then there exist positive integers $m,n$, and $\delta>0$, such that, if;\\

$|\rho'(C_{\bar j,m,n})-\rho(C_{\bar j,m,n})| <\delta$\\

for all sets $C_{\bar j,m,n}$ belonging to $C_{m,n}$, then;\\

$|\int_{[0,1]^{\mathcal N}}g d\rho' - \int_{[0,1]^{\mathcal N}}g d\rho| < \epsilon$\\

\end{lemma}

\begin{proof}
For a positive integer $n$, let $W_{n}$ consist of the inverse images in $X$ (from the projection $\pi_{n}$) of open boxes in $[0,1]^{n}$, with rational corners. Let $W=\bigcup_{n\in\mathcal{N}}W_{n}$. It is clear that $W$  forms a countable basis for the topology on ${[0,1]}^{\mathcal N}$. Adapting the compactness argument, given above in Lemma \ref{eleven}, for any $\gamma>0$ and $g\in C(X)$, we can find a positive integer $n$, and finitely many sets $\{W_{1,n},\ldots,W_{r,n}\}$ in $W_{n}$, covering $X$, such that $|g(x)-g(y)|<\gamma$ for all $x,y$ in $W_{j,n}$, $1\leq j\leq r$. Now choose $m$ such that each set of the partition $C_{m,n}$ lies inside one of the $W_{j,n}$. Then $|g(x)-g(y)|<\gamma$ on each $C_{\bar j,m,n}$, belonging to $C_{m,n}$. Now, for given $\delta>0$, suppose we choose $\rho'$ such that $|\rho'(C_{\bar j,m,n})-\rho(C_{\bar j,m,n})| <\delta$, $(*)$. Then;\\

$|\int_{X}g d\rho'-\int_{X}g d\rho|=|\sum_{\bar j}\int_{C_{\bar j,m,n}}g  d\rho' - \sum_{\bar j}\int_{C_{\bar j,m,n}}g  d\rho|$\\

$\leq \sum_{\bar j}|\int_{C_{\bar j,m,n}}g  d\rho'-\int_{C_{\bar j,m,n}}g  d\rho|$, $(**)$\\

Without loss of generality, assuming $\rho'$ is positive, by definition of the integral, see \cite{Rud}, we have that;\\

$c_{\bar j}\rho'(C_{\bar j,m,n})\leq \int_{C_{\bar j,m,n}}g  d\rho' \leq d_{\bar j}\rho'(C_{\bar j,m,n})$\\

$c_{\bar j}\rho(C_{\bar j,m,n})\leq \int_{C_{\bar j,m,n}}g  d\rho \leq d_{\bar j}\rho(C_{\bar j,m,n})$\\

where $c_{\bar j}=inf_{C_{\bar j,m,n}}g$ and $d_{\bar j}=sup_{C_{\bar j,m,n}}g$. Then;\\

$c_{\bar j}\rho'(C_{\bar j,m,n})-d_{\bar j}\rho(C_{\bar j,m,n})\leq \int_{C_{\bar j,m,n}}g  d\rho' - \int_{C_{\bar j,m,n}}g d\rho$\\

$\leq d_{\bar j}\rho'(C_{\bar j,m,n})-c_{\bar j}\rho(C_{\bar j,m,n})$\\

Therefore, again, without loss of generality;\\

$|\int_{C_{\bar j,m,n}}g  d\rho'-\int_{C_{\bar j,m,n}}g  d\rho|$\\

$\leq (d_{\bar j}-c_{\bar j})\rho'(C_{\bar j,m,n})+|c_{\bar j}||\rho'(C_{\bar j,m,n})-\rho(C_{\bar j,m,n})|\leq \gamma\rho'(C_{\bar j,m,n}) + |c_{\bar j}|\delta$\\

 $(***)$\\

By $(*)$, $\rho'(X)=\sum_{\bar j}\rho'(C_{\bar j,m,n})\leq \sum_{\bar j}\rho(C_{\bar j,m,n})+\delta m^{n}=1+\delta m^{n}$, so using $(**)$, $(***)$, and the fact that $|g|\leq M$;\\

$|\int_{X}g d\rho'-\int_{X}g d\rho|\leq\gamma(1+\delta m^{n})+\delta Mm^{n}$\\

So if we choose $0<\gamma<{\epsilon\over 2}$ and $0<\delta<{\epsilon\over 2(\gamma+M)m^{n}}$, we obtain;\\

$|\int_{X}g d\rho'-\int_{X}g d\rho| < \epsilon$\\

as required.

\end{proof}

We finally claim;\\

\begin{theorem}
\label{fifteen}
If $C_{m,n}$ is a partition, as in Definition \ref{thirteen} and $\delta>0$, then there exists a periodic element $\beta$, such that;\\

$|\rho_{\beta}(C_{\bar j,m,n})-\rho(C_{\bar j,m,n})| <\delta$\\

for all sets $C_{\bar j,m,n}$ belonging to $C_{m,n}$.

\end{theorem}

\begin{proof}
Let $\Sigma=\{{1\over 2m},{3\over 2m},\ldots,{2m-1\over 2m}\}$. Define $\kappa:\Sigma^{n}\rightarrow{\mathcal{R}}$ by;\\

$\kappa(({{2j_{0}+1}\over 2m},\ldots,{{2j_{n-1}+1}\over 2m}))=\rho(C_{\bar j,m,n})$\\

As $C_{m,n}$ is a partition of $X$ and $\rho$ is a probability measure, $\kappa$ is a probability measure on $\Sigma^{n}$. Moreover, using the partition property and the fact that $\rho$ is $\sigma$-invariant;\\

$\sum_{\xi_{0}\in\Sigma}\kappa((\xi_{0},\ldots,\xi_{n-1}))=\rho(\pi_{n}^{-1}([0,1]\times E_{j_{1},m}\times\ldots\times E_{j_{n-1},m}))$\\

$=\rho(\pi_{n}^{-1}(E_{j_{1},m}\times\ldots\times E_{j_{n-1},m}\times [0,1]))$\\

$=\sum_{\xi_{0}\in\Sigma}\kappa((\xi_{1},\ldots,\xi_{n-1},\xi_{0}))$ $(*)$\\

Now let $N>0$ be a sufficiently large positive integer, then we claim that we can find a probability measure $\kappa'$ on $\Sigma^{n}$ such that;\\

$(i)$. $|\kappa'(\bar{\xi})-\kappa(\bar{\xi})|<\delta$\\

$(ii)$. The condition $(*)$ still holds.\\

$(iii)$. $N\kappa'(\bar{\xi})$ is a non-negative integer, for all $\bar{\xi}\in\Sigma^{n}$\\

This follows from a simple linear algebra argument. We can identify the set of real measures on $\Sigma^{n}$ with the real vector space $V$ of dimension $m^{n}$. The condition $(*)$ then defines a subspace $W\subset V$. The condition of being a probability measure requires that;\\

$\sum_{\xi_{0},\ldots,\xi_{n-1}\in\Sigma^{n}}\kappa((\xi_{1},\ldots,\xi_{n-1},\xi_{0}))=1$, $(**)$\\

which defines an affine space $S_{aff}\subset V$. $S_{aff}\cap W$ contains a rational point $q$, corresponding to the probability measure with coordinates $m^{-n}$. It is straightforward to see that $(S_{aff}\cap W)=[(S_{aff}-q)\cap W]+q$. Moreover, $(S_{aff}-q)\cap W$ is a vector space defined by rational coefficients, so it has a rational basis. This shows that rational points are dense in $S_{aff}\cap W$. We can, without loss of generality, assume that all the coordinates of $\kappa$ are strictly greater than zero. If not, consider instead the space $S_{aff}\cap W\cap W'$, where $W'=Ker(\pi)$ is the kernel of the projection onto the non-zero coordinates of $\kappa$. The same argument shows that rational points are dense in $S_{aff}\cap W\cap W'$. We can now obtain a probability measure $\kappa'$, satisfying conditions $(i)-(iii)$, by finding a rational vector sufficiently close to $\kappa$ in $S_{aff}\cap W$, and choosing $N$ large enough.\\

Now take a longest sequence $\{\xi^{0},\ldots,\xi^{r-1}\}$ of elements in $\Sigma^{n}$, such that;\\

$(1)$. $(\xi_{1}^{i},\ldots,\xi_{n-1}^{i})=(\xi_{0}^{i+1},\ldots,\xi_{n-2}^{i})$.\\

$(2)$. $Card(\{i:0\leq i<r, \xi^{i}=\xi\})\leq N\kappa'(\xi)$ for any $\xi\in\Sigma^{n}$\\

where $\xi^{i}=(\xi_{0}^{i},\ldots,\xi_{n-1}^{i})$, for $0\leq i\leq r$, and $\xi^{r}=\xi^{0}$.\\

Then, by graph theoretical considerations, (\footnote{The graph theory argument proceeds as follows. We construct a tree. For every $\xi'\in\Sigma^{n-1}$, where $\xi'=(\xi_{1},\ldots,\xi_{n-1})$, associate a vertex $v_{\xi'}$ (the trunk). Similarly, for every $\xi\in\Sigma^{n}$, where $xi=(\xi_{0},\ldots,\xi_{n-1})$, associate two vertices $l_{\xi}$ (left) and $r_{\xi}$ (right). Attach the vertex $l_{\xi}$ to $v_{\xi'}$ iff $\pi(\xi)=\xi'$, where $\pi$ is the projection onto the last $n-1$ coordinates, and, attach $l_{\xi}$ to $v_{\xi'}$. iff $\pi'(\xi)=\xi'$, where $\pi'$ is the projection onto the first $n-1$ coordinates. In this way, we obtain a tree, having $m^{n-1}(2m+1)$ vertices, $m^{n-1}(2m)$ branches, and $m_{n-1}$ components. Each element $\xi\in \Sigma^{n}$ corresponds to two vertices, one on the left and one on the right of the tree. Now attach weights $m_{\xi}=n_{\xi}$ to the left vertices and right vertices respectively, by assigning the vertices $l_{\xi}$ and $r_{\xi}$, the weights $m_{\xi}=N\kappa'(\xi)$ and $n_{\xi}=N\kappa'(\xi)$ respectively. Observe that, by the condition $(*)$ in the main text, for any given $\xi'$;\\\

$m_{\xi'}=\sum_{\xi\in\Sigma^{n}:\pi(\xi)=\xi'}m_{\xi}=n_{\xi'}=\sum_{\xi\in\Sigma^{n}:\pi'(\xi)=\xi'}n_{\xi}$ $(\dag)$\\

Now, given a sequence $\{{\xi}^{0},{\xi}^{1},\ldots,{\xi}^{k}\}$ of elements in $\Sigma^{n}$, where $\xi^{i}=(\xi^{i}_{0},\ldots, \xi^{i}_{n-1})$, for $0\leq i\leq k$, we attach sets $L_{\xi}$ to each vertex $l_{\xi}$, by requiring that, $\xi^{i}\in L_{\xi}$ iff $\xi^{i}=\xi$, and, similarly, we attach sets $R_{\xi}$ to each vertex $r_{\xi}$. We call a sequence allowed if $(i)$. For each ${\xi}\in\Sigma^{n}$, $Card(L_{\xi})=Card(R_{\xi})\leq m_{\xi}=n_{\xi}$ and $(ii)$. For each $1\leq i\leq k$, if $\xi^{i}$ appears in the set $R_{\xi}$, then $\xi^{i-1}$ appears in a set $L_{\xi''}$, where $l_{\xi''}$ and $r_{\xi}$ are attached to the same vertex $v_{\xi'}$, so that $\pi(\xi'')=\pi'(\xi)=\xi'$. Clearly, all allowed sequences are bounded in length by $N\kappa'(X)$, so there exists a longest allowed sequence $s=(\xi^{i})_{0\leq i\leq t}$. Let $\xi^{t}$ be the final element in the sequence, and suppose that $\xi^{t}\in L_{\xi''}$, then, we claim that $\xi^{0}$ belongs to a set $R_{\xi}$, where $\pi(\xi'')=\pi'(\xi)=\xi'$, $(\dag\dag)$. If not, all such sets $R_{\xi}$, with $\pi'(\xi)=\pi(\xi'')$, consists of elements $\xi^{i}$ with $i\geq 1$. If, for one of these sets $R_{\xi}$, $Card(R_{\xi})\lneq n_{\xi}$, then we can extend the sequence by setting ${\xi}^{t+1}=\xi$, clearly such a sequence is allowed, contradicting maximality. So we can assume that $Card(R_{\xi})= n_{\xi}$. By condition $(ii)$, for every element $\xi^{i}$, $i\geq 1$, appearing in $R_{\xi}$, there exists an element $\xi^{i-1}$ appearing in an $L_{\xi''}$, with $\pi(\xi'')=\pi({\xi}^{t})$. This provides a total of $w+1$ elements appearing in such $L_{\xi''}$, where $w=\sum_{\xi\in\Sigma^{n}:\pi'(\xi)=\xi'}n_{\xi}$. By $(\dag)$, this is greater than
$\sum_{\xi\in\Sigma^{n}:\pi(\xi)=\xi'}m_{\xi}$. Clearly, this contradicts condition $(i)$ of an allowed path. Hence, $(\dag\dag)$ is shown. Observe also that if $\xi'\in\Sigma^{n-1}$, and $s_{r,\xi'}$ denotes the total number of elements from the sequence $s$, appearing in sets to the right of $\xi'$, $s_{l,\xi'}$, to the left, then $s_{l,\xi'}=s_{r,\xi'}$, In particular, by $(\dag)$, $m_{\xi'}-s_{l,\xi'}=n_{\xi'}-s_{r,\xi'}\geq 0$, so the number of "vacant slots" (if there are any), is the same on both sides of a given $\xi'$, $(\dag\dag\dag)$. In order to see this, we can, without loss of generality, assume that $\pi'(\xi^{0})\neq\xi'$, then just note that an element $\xi^{i+1}$ belongs to a set on the right of $\xi'$ iff $\xi^{i}$ belongs to a set on the left of $\xi'$, by condition $(ii)$ of an allowed path. We now claim that for all $\xi\in\Sigma^{n}$, $Card(R_{\xi})=n_{\xi}$, $(\dag\dag\dag\dag)$, (so there are no vacant slots). We have already shown this in the particular case when $\pi'(\xi)=\pi'(\xi^{0})$. We define an element $\xi$ to be cyclic if $\pi(\xi)=\pi'(\xi)$,
so cyclic elements are just constant sequences. We define an element $\xi$ to be free if $Card(R_{\xi})\lneq n_{\xi}$.
No free cyclic element $\xi_{cyc}$ can encounter the sequence $s$, for suppose that there exists a $\xi^{i}$, for some $0\leq i\leq t$, with $\pi(\xi^{i})=\pi'(\xi_{cyc})$, then we can extend the sequence $s$ to $s'=\{\xi^{0},\ldots,\xi^{i},\xi_{cyc},\xi^{i+1},\ldots,\xi^{t}\}$, and still obtain an allowed path, contradicting maximality. So we have that, if $\xi$ is free cyclic, with $\pi_{\xi}=\xi'$, then $s_{l,\xi'}=s_{r,\xi'}=0$, $(\dag\dag\dag\dag\dag)$. Now suppose there exists a free element $\xi_{free}$. Choose the largest $k$, with $0\leq k\leq t$, such that $\xi^{k}$ appears in $L_{\xi''}$ with $\pi(\xi'')=\pi'(\xi_{free})$, $(\sharp)$. As we have observed, $k\lneq t$. We construct a forward path from $\xi_{free}$ as follows. Define $\eta^{0}=\xi_{free}$, add the element $\eta^{0}$ to $R_{\xi_{free}}$ and $L_{\xi_{free}}$, and call the new sets $R_{0,\xi}$ and $L_{0,\xi}$, for $\xi\in \Sigma^{n}$. Having defined $\eta^{j}$, there are four cases. If $\pi(\eta^{j})=\pi'(\eta^{0})$, terminate the sequence. Otherwise, if $\pi(\eta^{j})=\pi(\xi_{cyc})$ for some cyclic element with $Card(R_{j,\xi_{cyc}})\lneq n_{\xi_{cyc}}$, then define $\eta^{j+1}=\xi_{cyc}$, add the element $\eta^{j+1}$ to $R_{j,\xi_{cyc}}$ and $L_{i,\xi_{cyc}}$, calling the new sets $R_{j+1,\xi}$ and $L_{j+1,\xi}$, for $\xi\in \Sigma^{n}$. If there is no such cyclic element, and there exists a free element $\xi'$ with $\pi(\eta^{j})=\pi'(\xi')$ and $Card(R_{j,\xi'})\lneq n_{\xi'}$, then define $\eta^{j+1}=\xi'$ (so there is some choice here), and, as before, redefine the sets $R_{j,\xi}$ and $L_{j,\xi}$ to $R_{j+1,\xi}$ and $L_{j+1,\xi}$, for $\xi\in \Sigma^{n}$. If there is no free element of this form, then terminate the sequence. It is straightforward to see, using $(\dag\dag\dag)$, $(\dag\dag\dag\dag\dag)$, and the fact that $\eta^{0}$ is not cyclic, that the sequence $\{\eta^{0},\ldots,\eta^{j}\}$ terminates after a finite number of steps $l$, with $l>0$, and $\pi(\eta^{l})=\pi'(\eta^{0})$. Moreover, for all $k<i<t$, and $0\leq j\leq l$, we have that $\pi(\xi^{i})\neq\pi'(\eta^{j})$, by $(\sharp)$. Hence, we can construct an allowed sequence $s''=\{\xi^{0},\ldots,\xi^{k},\eta^{0},\ldots,\eta^{l}\xi^{k+1},\ldots,\xi^{t}\}$, contradicting maximality of $s$. This shows $(\dag\dag\dag\dag)$. It is clear that the sequence  $s'''=\{\xi^{0},\ldots,\xi^{r-1}\}$, as defined in the main text, is a longest allowed sequence, as defined in this footnote, using $(\dag\dag)$. Hence, by $(\dag\dag\dag\dag)$, we have equality in $(2)$ as required.}), one can show that equality holds in the above inequality in $(2)$, for any $\xi\in\Sigma^{n}$, $(***)$. Now let $\beta$ be the periodic element in $[0,1]^{\mathcal N}$, with period $n+r-1$, defined by;\\

$(\beta(0),\beta(1),\ldots,\beta(n+r-2))=(\xi^{0}_{0},\xi^{0}_{1},\ldots,\xi^{0}_{n-1},\xi^{1}_{n-1},\xi^{2}_{n-1},\ldots,\xi^{r-1}_{n-1})$\\

By $(i)$, it is sufficient to prove that, for each $\bar j\in m^{n}$;\\

$|\rho_{\beta}(C_{\bar j,m,n})-\kappa'(\xi_{\bar j})|<\epsilon$, $(****)$,\\

 where $\epsilon=min_{\bar i}(\delta-|\kappa'(\xi_{\bar i})-\kappa(\xi_{\bar i})|)$, and $\xi_{\bar j}$ is the unique element of $\Sigma^{n}$ lying inside $C_{\bar j,m,n}$. By definition of $\rho_{\beta}$, $\rho_{\beta}(C_{\bar j,m,n})={c_{\bar j}\over n+r-1}$, where;\\

$c_{\bar j}=Card(\{k:0\leq k<n-r-1, \pi_{n}(\sigma^{k}(\beta))=\xi_{\bar j}\})$.\\

 By definition of $\beta$, and $(***)$, $c_{\bar j}={N\kappa'(\xi_{\bar j})+y\over n+r-1}$, where $0\leq y\leq n$. As $\kappa'$ is a probability measure, again by $(***)$, we have that $r-1=N$. Hence;\\

${c_{\bar j}\over n+r-1}={N\kappa'(\xi_{\bar j})+y\over N+n}=\kappa'(\xi_{\bar j})+{y-n\kappa'(\xi_{\bar j})\over N+n}$.\\

 Therefore,\\

 $|\rho_{\beta}(C_{\bar j,m,n})-\kappa'(\xi_{\bar j})|\leq{n\over N+n}<\epsilon$.\\

 if we choose $N$ sufficiently large. Hence, $(****)$ and the theorem are shown.

\end{proof}

We summarise what we have done;\\

\begin{theorem}
\label{sixteen}
The Ergodic Theorem \ref{one} holds and admits a non-standard proof.

\end{theorem}

\begin{proof}

Combine Theorems \ref{three},\ref{nine},\ref{fifteen}, and Lemmas \ref{four},\ref{six},\ref{seven},\ref{eleven},\ref{twelve},\ref{fourteen}.\\

\end{proof}
\begin{rmk}
\label{seventeen}
There are some outstanding questions in Ergodic Theory, which one might hope to solve using nonstandard methods, similar to the above. One of these is Ornstein's Isomorphism Theorem, I hope to investigate this direction further.
\end{rmk}

\end{section}
\begin{section}{Appendix}
\begin{theorem}
\label{appone}
 Suppose $g:X\rightarrow{\mathcal R}$ is integrable with respect to $\mu_{L}$, $\mu_{L}(X)<\infty$, and $\epsilon>0$ is standard, then there exist $F,G:X\rightarrow^{*}{\mathcal R}$, which are $\mathfrak{A}$-measurable, such that;\\

 (i). $G\leq g\leq F$.\\

 (ii). $|\int_{A}g d\mu_{L} -\int_{A} G d\nu|<\epsilon$, $|\int_{A}g d\mu_{L} -\int_{A} F d\nu|<\epsilon$\\

 for all $A\in\mathfrak{A}$.\\

 \end{theorem}

\begin{proof}
Consider, first, the case when $g\geq 0$.\\

Upper Bound. As $g$ is integrable, by Theorem 3.31 of \cite{dep}, it has an $S$-integrable lifting $F'$, such that ${^{\circ}F'}=g$ a.e $\mu_{L}$, and;\\

${^{\circ}\int_{X}F' d \nu}=\int_{X} g d\mu_{L}$\\

Without loss of generality, we cam assume that $F'\geq 0$. Now let $\epsilon>0$ be given and choose $\delta>0$ such that $\mu_{L}(X)\delta <{\epsilon\over 2}$. Then $F'+\delta$ is $S$-integrable and $F'+\delta\geq f$ a.e $\mu_{L}$, $(*)$, $F'+\delta>0$. Moreover;\\

${^{\circ}\int_{X} (F'+\delta) d\nu}=\int_{X} g d\mu_{L} + \delta\mu_{L}(X) < C+{\epsilon\over 2}$, $(**)$\\

where $C=\int_{X} g d\mu_{L}$. Let $N\in\mathfrak{M}_{L}$, with $\mu_{L}(N)=0$, such that $(*)$ holds on $N^{c}$. Let $N_{n}=N\cap g^{-1}((n-1,n])$, for $n\in\mathcal{N}_{>0}$, $N_{0}=N\cap g^{-1}(0)$. Then $N=\bigcup_{n\geq 0}N_{n}$, and $\mu_{L}(N_{n})=0$. By Lemma 3.15 (3.4(i)) of \cite{dep}, we can choose $U_{n}\supset N_{n}$, with $U_{n}\in\mathfrak{A}$, such that $\mu_{L}(U_{n})<{\epsilon\over 4(n+1)^{3}}$. Inductively, define $F_{0}=F'+\delta$, and, having defined $F_{n}$, let $F_{n+1}=F_{n}$ on $U_{n+1}^{c}$, and $F_{n+1}=F_{n}+n+1$ on $U_{n+1}$. Then $\{F_{n}\}$ is an increasing sequence of $\mathfrak{A}$-measurable functions. Moreover;\\

$\int_{X} F_{n+1} d\nu$\\

$=\int_{U_{n+1}^{c}} F_{n} d\nu + \int_{U_{n+1}} (F_{n}+(n+1)) d\nu$\\

$\simeq\int_{X} F_{n} d\nu + (n+1)\mu_{L}(U_{n+1})$\\

$<\int_{X} F_{n} d\nu + {\epsilon\over 4(n+1)^{2}}$\\

$\int_{X}F_{n} d\nu < C+{\epsilon\over 2} +\sum_{m=1}^{n}{\epsilon\over 4m^{2}}<C+\epsilon$ (using $(**)$)\\

We clearly have that for all $x\in N_{n}$, $g(x)\leq F_{n}$. Now, by countable comprehension, we can find an internal sequence $\{F_{n}\}_{n\in{^{*}\mathcal{N}}}$ extending the sequence $\{F_{n}\}_{n\in\mathcal{N}}$. By overflow, there exists an infinite $\omega$, such that $F_{n}\leq F_{\omega}$, for all $n\in\mathcal{N}$, $F_{\omega}>0$, and;\\

$\int_{X} F_{\omega} d\nu < C+\epsilon$, $(\dag)$\\

Clearly $g(x)\leq F_{\omega}(x)$, for all $x\in X$. Now, if $A\in\mathfrak{A}$, with;\\

$\int_{A} F_{\omega} d\nu-\int_{A}g d\mu_{L}>\epsilon$\\

then, using Theorem 3.16 of \cite{dep};\\

$\int_{X} F_{\omega} d\nu$\\

$=\int_{A} F_{\omega} d\nu +\int_{A^{c}} F_{\omega} d\nu$\\

$>\epsilon +\int_{A}g d\mu_{L} +\int_{A^{c}}g d\mu_{L}=C+\epsilon$\\

contradicting $(\dag)$. Setting $F=F_{\omega}$ gives an upper bound.\\

Lower Bound. Again choose $\delta>0$, with $\mu_{L}(X)\delta<{\epsilon\over 2}$. Let $F'$ be as before, then $F'-\delta$ is $S$-integrable, $F'-\delta\leq g$ a.e $\mu_{L}$, and:\\

$\int_{X} (F'-\delta) d\nu > C-{\epsilon\over 2}$\\

 Again choose $N$, with $\mu_{L}(N)=0$, such that $F'-\delta\leq g$ on $N^{c}$. Using Lemma 3.15(3.4(i)) of \cite{dep} again, we can choose a decreasing sequence of sets $\{U_{n}\}_{n\in{\mathcal N}_{>0}}$, belonging to $\mathfrak{A}$, with $U_{n}\supset N$, and $\mu_{L}(U_{n})<{1\over n}$. By $S$-integrability;\\

 ${^{\circ}\int_{U_{n}}(F'-\delta) d\nu}=\int_{U_{n}}{^{\circ}(F'-\delta)} d\mu_{L}$\\

 and;\\

 $lim_{n\rightarrow\infty}(\int_{U_{n}}{^{\circ}(F'-\delta)} d\mu_{L})=0$\\

 by the DCT, as ${^{\circ}(F'-\delta)}\chi_{U_{n}}$ converges to $0$ a.e $\mu_{L}$. Hence, for sufficiently large $n$, we can assume that;\\

 $\int_{U_{n}}(F'-\delta) d\nu <{\epsilon\over 2}$\\

 Now let $G=(F'-\delta)$ on $U_{n}^{c}$, and $G=0$ on $U_{n}$. Clearly $G(x)\leq g(x)$, for all $x\in X$. Moreover;\\

 $\int_{X} G d\nu$\\

 $=\int_{U_{n}^{c}}(F'-\delta) d\nu$\\

 $=\int_{X}(F'-\delta) d\nu -\int_{U_{n}} (F'-\delta) d\nu > C-\epsilon$\\

 The same argument as above shows that, for all $A\in\mathfrak{A}$;\\

 $\int_{A}g d\mu_{L}-\int_{A} G d\nu \leq \epsilon$\\

 Hence, $G$ is a lower bound.\\

Now, if $g$ is integrable $\mu_{L}$, we can write $g=g^{+}-g^{-}$, with $\{g^{+},g^{-}\}$ integrable $\mu_{L}$. Choosing $G\geq g^{+}$ and $H\leq g^{-}$, $G-H\geq (g^{+}-g^{-})=g$, choosing $G'\leq g^{+}$ and $H'\geq g^{-}$, $G'-H'\leq (g^{+}-g^{-})=g$, and, clearly, we can obtain the integral condition, using ${\epsilon\over 2}$.\\

\end{proof}
\end{section}

\end{document}